\documentclass[a4paper, oneside,12pt]{amsart}
\usepackage{graphicx}
\usepackage[english]{babel}
\usepackage{mathrsfs,amssymb}
\usepackage{mathtools}
\usepackage[colorlinks, citecolor = blue]{hyperref}
\usepackage{tcolorbox}
\usepackage[shortlabels]{enumitem}
\setlist[itemize]{leftmargin=25pt}
\setlist[enumerate]{leftmargin=25pt}
\usepackage{pgfplots}
\pgfplotsset{every axis/.append style={tick label style={/pgf/number format/fixed},font=\scriptsize,ylabel near ticks,xlabel near ticks,grid=major}}

\usepackage{geometry}
\makeatletter
\newcommand{\leqnomode}{\tagsleft@true}
\newcommand{\reqnomode}{\tagsleft@false}
\makeatother

\usepackage{todonotes}

\newtheorem{theorem}{Theorem}[section]
\newtheorem{lemma}[theorem]{Lemma}

\theoremstyle{definition}

\theoremstyle{remark}

\numberwithin{equation}{section}

\DeclareMathOperator*{\esssup}{ess\,sup}
\DeclareMathOperator*{\essinf}{ess\,inf}

\newcommand{\ddn}{\mathrm{d}}

\let \a=\alpha

\def\avint_#1{\mathchoice{\mathop{\kern 0.2em\vrule width 0.6em height 0.69678ex depth -0.58065ex \kern -0.8em \intop}\nolimits_{\kern -0.4em#1}}{\mathop{\kern 0.1em\vrule width 0.5em height 0.69678ex depth -0.60387ex \kern -0.6em \intop}\nolimits_{#1}} {\mathop{\kern 0.1em\vrule width 0.5em height 0.69678ex depth -0.60387ex \kern -0.6em \intop}\nolimits_{#1}} {\mathop{\kern 0.1em\vrule width 0.5em height 0.69678ex depth -0.60387ex \kern -0.6em \intop}\nolimits_{#1}}}

\allowdisplaybreaks

\begin{document}
\title[Sharpness of some quantitative Muckenhoupt--Wheeden inequalities]
{On the sharpness of some quantitative Muckenhoupt--Wheeden inequalities}

\author[A.K. Lerner]{Andrei K. Lerner}
\address[A.K. Lerner]{Department of Mathematics,
Bar-Ilan University, 5290002 Ramat Gan, Israel}
\email{lernera@math.biu.ac.il}

\author[K. Li]{Kangwei Li}
\address[K. Li]{Center for Applied Mathematics, Tianjin University, Weijin Road 92, 300072 Tianjin, China}
\email{kli@tju.edu.cn}

\author[S. Ombrosi]{Sheldy Ombrosi}
\address[S. Ombrosi]{Departamento de Análisis Matemático y Matemática Aplicada\\ Universidad Complutense (Spain) \&
Departamento de Matemática e Instituto de Matemática. Universidad Nacional del Sur - CONICET Argentina}
\email{sombrosi@ucm.es}

\author[I.P. Rivera-R\'ios]{Israel P. Rivera-R\'ios}
\address[I.P. Rivera-R\'ios]{Departamento de An\'alisis Matem\'atico, Estad\'istica e Investigaci\'on Operativa
y Matem\'atica Aplicada. Facultad de Ciencias. Universidad de M\'alaga (M\'alaga, Spain).}
\email{israelpriverarios@uma.es}

\thanks{The first author was supported by ISF grant no. 1035/21. The second author was supported by the National Natural Science Foundation of China through project numbers 12222114 and 12001400. The fourth author was supported by Spanish Ministerio de Ciencia e Innovaci\'on grant PID2022-136619NB-I00 and by Junta de Andaluc\'ia grant FQM-354.}

\begin{abstract} 
In the recent work \cite{CUIMPRR21} it was obtained that
\[|\{x\in{\mathbb{R}^d}:w(x)|G(fw^{-1})(x)|>\a\}|\lesssim\frac{[w]_{A_1}^2}{\a}\int_{{\mathbb{R}^d}}|f|\ddn x\]
both in the matrix and scalar settings, where $G$ is either the Hardy--Littlewood maximal function or any Calder\'on--Zygmund operator.
In this note we show that the quadratic dependence on $[w]_{A_1}$ is sharp. This is done by constructing a sequence of scalar-valued weights with blowing up characteristics so that the corresponding bounds for the Hilbert transform and maximal function are exactly quadratic.
\end{abstract}


\maketitle
\section{Introduction}
Recall that a non-negative locally integrable function $w$ satisfies the $A_1$ condition if there exists a constant $C>0$ such that for every cube $Q\subset {\mathbb R}^d$,
$$\frac{1}{|Q|}\int_Qw\le C\essinf_{Q}w.$$
The smallest constant $C$ for which this property holds is denoted by~$[w]_{A_1}$.

In the 70s, Muckenhoupt and Wheeden \cite{MW77} established weighted
weak type $(1,1)$ bounds of the form
\begin{equation}
|\{x\in{\mathbb{R}}:w(x)|T(fw^{-1})(x)|>\a\}|\lesssim\frac{C_{w}}{\a}\int_{{\mathbb{R}}}|f|\ddn x,\label{weakscalar}
\end{equation}
where $T$ is either the Hilbert transform $H$ or the Hardy--Littlewood
maximal operator~$M$. They showed as well that $w\in A_{1}$ is a
sufficient condition for those inequalities to hold, even though it
is not necessary. Muckenhoupt and Wheeden observed that (\ref{weakscalar})
could be regarded as a first step to settle inequalities of the form
\begin{equation}
u^{1-r}\left(\left\{u^r|Tf|>\a\right\} \right)\lesssim\frac{1}{\a}\int_{\mathbb{R}}|f|w\,\ddn x\label{eq:2wMW},
\end{equation}
where $u,w$ are non-negative functions and $r\in[0,1]$. Note that for $u=w$,
for $r=0$ this inequality is the standard weak type inequality and in the case $r=1$ it reduces to (\ref{weakscalar}).
Their idea was to combine (\ref{eq:2wMW}) with interpolation with change
of measures in order to obtain two weighted estimates.

Pushing forward that idea, Sawyer \cite{S85} showed
that
\begin{equation}
uv\left(\left\{ \frac{M(fv)}{v}>\a\right\} \right)\lesssim C_{u,v}\frac{1}{\a}\int_{\mathbb{R}}|f|uv\,\ddn x,\label{eq:S}
\end{equation}
where $u,v\in A_{1}$. This estimate combined with interpolation with
change of measures allowed him to reprove Muckenhoupt's maximal theorem.

Since the aforementioned papers a number of works have been devoted to estimates related
to the ones above, that are known in the literature as mixed weak
type estimates. Some worth mentioning are \cite{CUMP05} where
(\ref{eq:S}) is extended to higher dimensions and further operators
such as Calderón--Zygmund operators via extrapolation, or \cite{LOP19}
where it is shown that $u\in A_{1}$ and $v\in A_\infty$ is sufficient for (\ref{eq:S})
to hold.

In terms of quantitative estimates  for $C_{u,v}$ in (\ref{eq:S}), and up to very recently for $C_w$ in \eqref{weakscalar}, as we will mention soon, not very much is known. Some results
are provided in the aforementioned work \cite{LOP19} or in some other
papers such as \cite{OPR} or \cite{CRR} for $C_{u,v}$ in (\ref{eq:S}). The purpose of this note is to provide some insight
on $C_w$ in \eqref{weakscalar}. However, before presenting our main result we would like to connect this problem
with the matrix weighted setting. We devote the following lines to that purpose.

In the last years quantitative matrix weighted estimates have attracted the attention of a number of authors. Up until now  only few sharp quantitative results in the matrix weight setting are known. Among them
the sharp $L^p(W)$ bounds for the maximal operator~\cite{IM19}, the sharp $L^2(W)$ bound for the square function \cite{HPV19}, and also
the sharp $L^p(W)$ bounds in terms of the $[W]_{A_q}$ constants with $1\leq q<p$ obtained in \cite{IPRR} for the maximal operator, Calder\'on--Zygmund operators and commutators.
Very recently Domelevo, Petermichl, Treil and Volberg \cite{DPTV} showed the sharpness of the $L^2(W)$ bound by $[W]_{A_2}^{3/2}$ for Calder\'on--Zygmund operators obtained previously in \cite{NPTV17}.

Making sense of endpoint matrix weighted estimates is a tricky problem. Quite recently, Cruz-Uribe et al. \cite{CUIMPRR21}
managed to obtain the first quantitative endpoint estimates in that setting. In order to state this result, we first give
several definitions.

Assume that $W$ is a matrix weight, that is, $W$ is an $n\times n$ self-adjoint matrix function with locally
integrable entries such that $W(x)$ is positive definite for a.e. $x\in {\mathbb R}^d$. Define the operator norm of $W$ by
$$\|W(x)\|:=\sup_{e\in {\mathbb C}^n:|e|=1}|W(x)e|.$$ We say that $W\in A_1$ if
$$[W]_{A_1}:=\sup_{Q}\esssup_{y\in Q}\frac{1}{|Q|}\int_Q\|W(x)W(y)^{-1}\|\ddn x<\infty.$$
It is easy to see that the matrix $A_1$ constant $[W]_{A_1}$ coincides with $[w]_{A_1}$ when $n=1$.

Given a matrix weight $W$, a vector-valued function $\vec{f}:{\mathbb R}^d\to {\mathbb C}^n$ and a Calder\'on--Zygmund operator $T$, define
$$T_W\vec{f}(x):=W(x)T(W^{-1}\vec{f})(x).$$
Next, define the maximal operator by
$$M_W\vec{f}(x):=\sup_{Q\ni x}\frac{1}{|Q|}\int_Q|W(x)W^{-1}(y)\vec{f}(y)|\ddn y.$$
The operators above have the obvious interpretation in the scalar setting.
\vskip 2mm
\noindent
{\bf Theorem A} (\cite{CUIMPRR21}). {\it We have
\begin{equation}\label{weakmatrix}
|\{x\in {\mathbb R}^d: |T_W\vec{f}(x)|>\a\}|\lesssim \frac{[W]_{A_1}^2}{\a}\int_{{\mathbb R}^d}|\vec{f}|\ddn x,
\end{equation}
and the same holds for $M_W$.}
\vskip 2mm

At this point we are in the position to state the main result of this note.

\begin{theorem}\label{mainresult}
In the scalar-valued setting the quadratic dependence on $[w]_{A_1}$ in (\ref{weakmatrix}) is sharp for $T_w$ and for $M_w$.
\end{theorem}

As a direct consequence of Lemma \ref{lem:connect} below, this result shows the sharpness of $[W]_{A_1}^2$ in Theorem A in the matrix setting as well.

An interesting phenomenon here is the contrast between the strong $L^2(W)$ and the weak $L^1(W)$ bounds for Calder\'on--Zygmund operators. As we mentioned above, the recent work {\cite{DPTV}} establishes the sharpness
of $[W]_{A_2}^{3/2}$ in the matrix setting. Comparing this with the linear $A_2$ bound in the scalar case \cite{H12}, we see that the sharp weighted $L^2$ bounds for Calder\'on--Zygmund operators are different in the matrix and scalar settings. However, Theorem \ref{mainresult} shows that the sharp weighted weak $L^1$ bounds are the same in both settings.

\section{Proof of Theorem \ref{mainresult}}
\subsection{Connection between the scalar and the matrix weighted estimates}
\begin{lemma}\label{lem:connect}
Assume that $G_{W}$ stands either for $T_{W}$ or for $M_{W}$. Then
if
\[
\left|\left\{ x\in\mathbb{R}^{d}\ :\ |G_{W}\vec{f}(x)|>\alpha\right\} \right|\leq c\varphi\left([W]_{A_{1}}\right)\frac{1}{\alpha}\int_{{\mathbb R}^d}|\vec{f}|\ddn x,
\]
we have that for every $w\in A_{1}$
\[
\left|\left\{ x\in\mathbb{R}^{d}\ :\ |J_{w}f(x)|>\alpha\right\} \right|\leq c\varphi\left([w]_{A_{1}}\right)\frac{1}{\alpha}\int_{{\mathbb R}^d}|f|\ddn x,
\]
where $J_{w}$ stands, respectively, for $wT(fw^{-1})$ or for $wM(fw^{-1})$.
\end{lemma}
\begin{proof}
Let $w\in A_{1}$. It is clear that $W=wI_{n}$ is a matrix $A_{1}$
weight. Furthermore,
\[
[W]_{A_{1}}=[w]_{A_{1}}.
\]
Now given a scalar function $f$, we build $\vec{f}=(f,0,\dots,0)^{t}.$
Note that for these choices of $\vec{f}$ and $W$, clearly
\[
|G_{W}\vec{f}(x)|=|J_{w}f(x)|.
\]
This ends the proof.
\end{proof}

\subsection{Construction of a family of weights providing the lower bound} We prove Theorem \ref{mainresult} via the following result.
\begin{theorem}\label{thm:m}
For any integer $N>20$, there exists a scalar weight $w\in A_1(\mathbb R)$ satisfying the following properties:
\begin{enumerate}
\item $\int_0^1w=1$;
\item $[w]_{A_1}\simeq N$;
\item $|\{x\in (1, \infty): w(x)> x\}|\gtrsim N^2$.
\end{enumerate}
\end{theorem}
Observe that Theorem \ref{thm:m} implies Theorem \ref{mainresult} immediately because if we take
$f=\chi_{[0,1]}$, then for $x>1$,
\[
Mf(x)= \frac 1x,\qquad Hf(x)= \int \frac{f(y)}{x-y} \ddn y >\frac 1x.
\]
Hence, if $T$ is either $M$ or $H$, then
\begin{align*}
\| w Tf\|_{L^{1,\infty}}&\ge \big|\{x\in (1, \infty): w(x)|Tf(x)|>1 \} \big|\\
& \ge \big|\{x\in (1, \infty): w(x)\cdot\frac 1x>1 \} \big|\\
&\gtrsim N^2 \simeq [w]_{A_1}^2 \|f\|_{L^1(w)}.
\end{align*}
The rest of this section will be devoted to proving Theorem \ref{thm:m}.
\begin{proof}[Proof of Theorem \ref{thm:m}]
For $k=2,3,\ldots, N$ we denote $J_k=[2^k, 2^{k+1})$. We will split $J_k$ into small intervals.
Set $I_k=[2^k, 2^k+k)$ and $L_k= J_k\setminus I_k=[2^k+k, 2^{k+1})$. Let $L_k^{-}$ and $L_k^{+}$ be the left and right halves of $L_k$, respectively. Next we define $(L_k^{-})^1$ to be the right half of $L_k^{-}$ and $(L_k^{+})^1$ the left half of $L_k^{+}$. Then
\begin{enumerate}
\item[$\bullet$]when $(L_k^{-})^j=[a_{k}^j, b_k^j)$ is defined, let $(L_k^{-})^{j+1}=[a_{k}^{j+1}, b_{k}^{j+1})$ satisfy that
\[
b_{k}^{j+1}=a_k^j,\qquad |(L_k^{-})^{j+1}|= \frac 12 |(L_k^{-})^j|;
\]
\item[$\bullet$]when $(L_k^{+})^j=[c_{k}^j, d_k^j)$ is defined, let $(L_k^{+})^{j+1}=[c_{k}^{j+1}, d_{k}^{j+1})$ satisfy that
\[
c_{k}^{j+1}=d_k^j,\qquad |(L_k^{+})^{j+1}|= \frac 12 |(L_k^{+})^j|.
\]
\end{enumerate}
The process is stopped when we have $(L_k^{-})^{k-1}$ and $(L_k^{+})^{k-1}$ defined, and we simply define
\[
(L_k^{-})^{k}= [2^k+k, 2^k+k+ \frac{|L_k^-|}{2^{k-1}}),\,\, (L_k^{+})^{k}= [2^{k+1}-\frac{|L_k^+|}{2^{k-1}}, 2^{k+1}).
\]Now we have split $J_k$ into disjoint intervals, namely,
\[
J_k= I_k  \cup   \mathop{\cup}_{j=1}^k  (L_k^{-})^j \cup  \mathop{\cup}_{j=1}^k (L_k^{+})^j.
\]

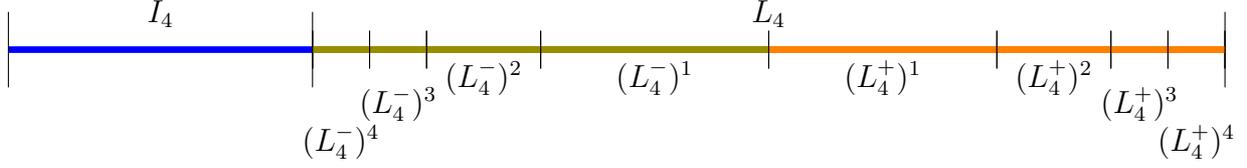
\begin{figure}[h!]
        \centering

         \makebox[\textwidth][c]{
\begin{tikzpicture}[scale=1]
    \draw (16,0.25)-- (32,0.25); 
     \draw[line width=2.5pt, blue] (16,0.25)-- (20,0.25);
    \foreach \x in {16,20,32} {
          \draw (\x,0.75) -- (\x,-0.25);

    }
    \draw (18,0.4) node[above] {$I_4$};
    \draw (26,0.4) node[above] {$L_4$};

    \draw[line width=2.5pt, olive] (20,0.25)-- (26,0.25);
    \draw[line width=2.5pt, orange] (26,0.25)-- (32,0.25);
     \foreach \x in {20,26,32} {
       \draw (\x,0.5) -- (\x,0);
    }

    \draw[line width=2.5pt, olive] (23,0.25)-- (26,0.25);
     \draw[line width=2.5pt, orange] (26,0.25)-- (29,0.25);
     \foreach \x in {23,26,29} {
        \draw (\x,0.5) -- (\x,0);
    }
     \draw (24.5,0.25) node[below] {$(L_4^-)^1$};
    \draw (27.5,0.25) node[below] {$(L_4^+)^1$};

     \draw[line width=2.5pt, olive] (21.5,0.25)-- (23,0.25);
     \draw[line width=2.5pt, orange] (29,0.25)-- (30.5,0.25);
     \foreach \x in {21.5,23,29,30.5} {
       \draw (\x,0.5) -- (\x,0);
    }
    \draw (22.25,0.25) node[below] {$(L_4^-)^2$};
    \draw(29.75,0.25) node[below] {$(L_4^+)^2$};
%
     \draw[line width=2.5pt, olive] (20.75,0.25)-- (21.5,0.25);
     \draw[line width=2.5pt, orange] (30.5,0.25)-- (31.25,0.25);
     \foreach \x in {21.5,20.75,31.25,30.5} {
        \draw (\x,0.5) -- (\x,0);
    }

     \draw (21.125,-0.1) node[below] {$(L_4^-)^3$};
    \draw (30.875,-0.1) node[below] {$(L_4^+)^3$};

 \draw[line width=2.5pt, olive] (20,0.25)-- (20.75,0.25);
     \draw[line width=2.5pt, orange] (31.25,0.25)-- (32,0.25);
     \foreach \x in {20,20.75,31.25,32} {
        \draw (\x,0.5) -- (\x,0);
    }

     \draw (20.375,-0.6) node[below] {$(L_4^-)^4$};
    \draw (31.625,-0.6) node[below] {$(L_4^+)^4$};
    \end{tikzpicture}
}
\caption{Component intervals of $J_4$.}
\end{figure}

Define
\[
w_k = 2^{k+1} \chi_{I_k}+ \sum_{j=1}^k 2^j \chi_{ (L_k^{-})^j  \cup  (L_k^{+})^j  }
\]
and our weight on $[0, 2^{N+2}]$ is
\begin{equation*}
w(x)=\begin{cases}
 \chi_{[0, 4)}(x)+ \sum\limits_{k=2}^N w_k(x), \quad & x\in [0, 2^{N+1}),\\
2^{N}, \quad & x=2^{N+1},\\
w(2^{N+2}-x), \quad & x\in [2^{N+1}, 2^{N+2}].
 \end{cases}
\end{equation*}

 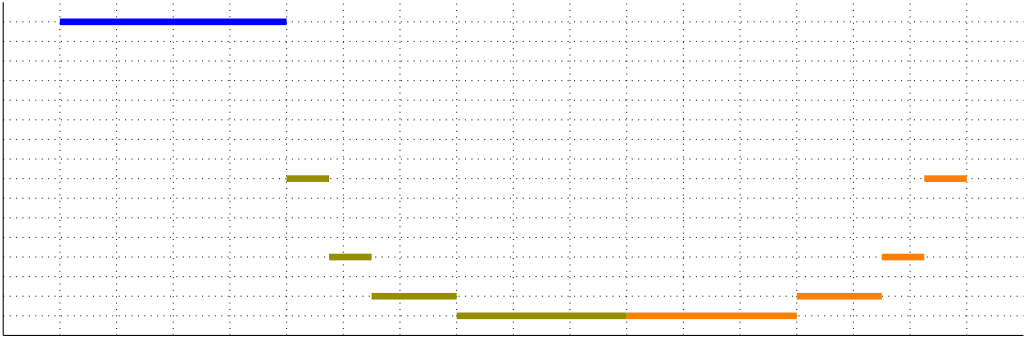
\begin{figure}[h!]
        \centering

         \makebox[\textwidth][c]{
 \begin{tikzpicture}
            \begin{axis}[
            width=15cm,height=6cm,
            ymin=0,ymax=34,xmin=15,xmax=33,
            xtick={0,1,...,32},
            ytick={0,2,...,32},
            y axis line style={draw opacity=0},
            grid style={dotted,gray!30!black},
            yticklabels={,,},
            xticklabels={,,},
            axis x line=center,
            axis y line=center,
            axis  line style={-},
            ticks=none
            ]
                \addplot[line width=2.5pt,olive,samples at={20,20.75}] {16};
                \addplot[line width=2.5pt,orange,samples at={31.25,32}] {16};
                \addplot[line width=2.5pt,blue,samples at={16,20}] {32};
                \foreach \j
                [evaluate=\j as \u using 4+16+(32-4-16)/2-(32-4-16)/2*(1-1/2^(\j+1)),
                evaluate=\j as \v using 4+16+(32-4-16)/2-(32-4-16)/2*(1-1/2^(\j)),
                evaluate=\j as \w using 2^(\j+1)]
                in {0,1,2} \addplot[line width=2.5pt,olive,samples at={\u,\v}] {\w};
                \foreach \j
                [evaluate=\j as \u using 4+16+(32-4-16)*(1-1/2^(\j+2)),
                evaluate=\j as \v using 4+16+(32-4-16)*(1-1/2^(\j+1)),
                evaluate=\j as \w using 2^(\j+1)]
                in {0,1,2} \addplot[line width=2.5pt,orange,samples at={\u,\v}] {\w};

            \end{axis}
        \end{tikzpicture}
}
\caption{Graph of $w_4$}
\end{figure}
 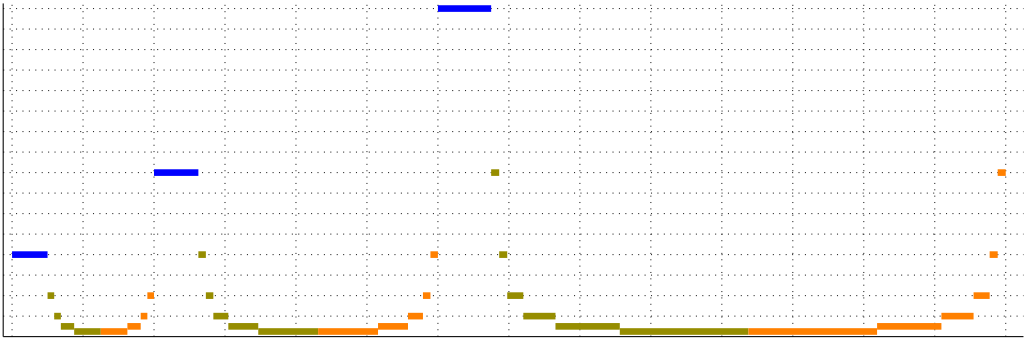
\begin{figure}[h!]
        \centering

         \makebox[\textwidth][c]{
        \begin{tikzpicture}
            \begin{axis}[width=15cm,height=6cm,
            ymin=0,ymax=130,xmin=15, xmax=130,
            xtick={0,8,...,128}, ytick={0,8,...,128},
            legend style={at={(0.9,0.9)},anchor=north},
            grid style={dotted,gray!30!black},
            yticklabels={,,},
            xticklabels={,,},
            axis x line=center,
            axis y line=center,
            y axis line style={draw opacity=0},
            axis  line style={-},
            ticks=none]
                \addplot[line width=2.5pt,olive,samples at={20,20.75}] {16};
                \addplot[line width=2.5pt,orange,samples at={31.25,32}] {16};
                \addplot[line width=2.5pt,blue,samples at={16,20}] {32};
                \foreach \j
                [evaluate=\j as \u using 4+16+(32-4-16)/2-(32-4-16)/2*(1-1/2^(\j+1)),
                evaluate=\j as \v using 4+16+(32-4-16)/2-(32-4-16)/2*(1-1/2^(\j)),
                evaluate=\j as \w using 2^(\j+1)]
                in {0,1,2} \addplot[line width=2.5pt,olive,samples at={\u,\v}] {\w};
                \foreach \j
                [evaluate=\j as \u using 4+16+(32-4-16)*(1-1/2^(\j+2)),
                evaluate=\j as \v using 4+16+(32-4-16)*(1-1/2^(\j+1)),
                evaluate=\j as \w using 2^(\j+1)]
                in {0,1,2} \addplot[line width=2.5pt,orange,samples at={\u,\v}] {\w};

               \addplot[line width=2.5pt,olive,samples at={37,37.84375}] {32};
                \addplot[line width=2.5pt,orange,samples at={63.15625,64}] {32};
                \addplot[line width=2.5pt,blue,samples at={32,37}] {64};
                \foreach \j
                [evaluate=\j as \u using 5+32+(64-5-32)/2-(64-5-32)/2*(1-1/2^(\j+1)),
                evaluate=\j as \v using 5+32+(64-5-32)/2-(64-5-32)/2*(1-1/2^(\j)),
                evaluate=\j as \w using 2^(\j+1)]
                in {0,1,2,3} \addplot[line width=2.5pt,olive,samples at={\u,\v}] {\w};
                \foreach \j
                [evaluate=\j as \u using  5+32+(64-5-32)*(1-1/2^(\j+2)),
                evaluate=\j as \v using 5+32+(64-5-32)*(1-1/2^(\j+1)),
                evaluate=\j as \w using 2^(\j+1)]
                in {0,1,2,3} \addplot[line width=2.5pt,orange,samples at={\u,\v}] {\w};

                \addplot[line width=2.5pt,olive,samples at={70,70.90625}] {64};
                \addplot[line width=2.5pt,orange,samples at={127.09375,128}] {64};
                \addplot[line width=2.5pt,blue,samples at={64,70}] {128};
                \foreach \j
                [evaluate=\j as \u using 6+64+(128-6-64)/2-(128-6-64)/2*(1-1/2^(\j+1)),
                evaluate=\j as \v using 6+64+(128-6-64)/2-(128-6-64)/2*(1-1/2^(\j)),
                evaluate=\j as \w using 2^(\j+1)]
                in {0,1,2,3,4} \addplot[line width=2.5pt,olive,samples at={\u,\v}] {\w};
                \foreach \j
                [evaluate=\j as \u using  6+64+(128-6-64)*(1-1/2^(\j+2)),
                evaluate=\j as \v using 6+64+(128-6-64)*(1-1/2^(\j+1)),
                evaluate=\j as \w using 2^(\j+1)]
                in {0,1,2,3,4} \addplot[line width=2.5pt,orange,samples at={\u,\v}] {\w};

            \end{axis}
        \end{tikzpicture}
        }
\caption{Joint graph of $w_4$, $w_5$ and $w_6$.}
\end{figure}

Finally we extend $w(x)$ from $[0, 2^{N+2}]$ to $\mathbb R$ periodically with period $2^{N+2}$.  Such a weight trivially satisfies $\int_0^1w=1$. Moreover, since $w(x)>x$ on $I_k$, we have
\begin{align*}
|\{x\in (1, \infty): w(x)> x\}|\ge \sum_{k=2}^N |I_k|=\sum_{k=2}^Nk\simeq N^2.
\end{align*}

Hence it remains to check that $[w]_{A_1} \simeq N$. Since $w$ is periodic on $\mathbb R$ and symmetrical on $[0, 2^{N+2}]$, it suffices to prove that
\[
\sup_{I \subset [0, 2^{N+1}]}\frac {w(I)}{|I| \essinf\limits_{x\in I}w(x)} \simeq N
\]
(we use the standard notation $w(E)=\int_Ew$).

Observe that $|L_k^{-}|=|L_k^{+}|=\frac{1}{2}(2^k-k)$. Further,
$$|(L_k^{-})^j|=|(L_k^{+})^j|=\frac{1}{2^{j+1}}(2^k-k),\quad j=1,\dots, k-1,$$
and
$$|(L_k^{-})^k|=|(L_k^{+})^k|=\frac{1}{2^k}(2^k-k).$$
Hence,
\begin{eqnarray*}
w_k(J_k)&=&2^{k+1}|I_k|+\sum_{j=1}^k2^j\big(|(L_k^{-})^j|+|(L_k^{+})^j|\big)\\
&=&2^{k+1}k+(k-1)(2^k-k)+2(2^k-k)\simeq k2^k.
\end{eqnarray*}
From this, when $I = [0, 2^{N+1}]$ we have
\begin{align*}
\frac {w(I)}{|I|}= 2^{-(N+1)}\Big( 4+\sum_{k=2}^N w_k(J_k) \Big) &\simeq  2^{-(N+1)} \Big( 4+\sum_{k=2}^N k 2^{k}\Big) \\&\simeq N = N \essinf\limits_{x\in I}w(x).
\end{align*}

Therefore, we are left to prove that for any $I\subset  [0, 2^{N+1}]$,
\begin{equation}\label{eq:a1up}
\frac{w(I)}{|I|}\lesssim N \essinf\limits_{x\in I}w(x).
\end{equation}

At this point we will make the following elementary observation. Our weight $w$ is a step function, and for each two adjacent intervals from its definition, the ``jump" of $w$ is at most 2.
Since the ``jumps" are multiplicative we have the following.

\vskip 1mm
{\bf Claim A.} If $I$ intersects at most $m$ intervals from the definition of $w$, then
$$
\frac{w(I)}{|I|}\le \max_{x\in I} w(x)\le 2^m \min_{x\in I} w(x)= 2^m\essinf\limits_{x\in I}w(x).
$$

\vskip 1mm
In what follows we will prove \eqref{eq:a1up} according to the size of $I$.

\vskip 1mm
{\bf Case 1}. $|I|\le 4$. In this case, note that in each $J_k$ $(k\ge 2)$,  $(L_k^-)^{k-1}, (L_k^-)^k$ and $(L_k^+)^{k-1},(L_k^+)^k$ are the smallest intervals, and
\[
|(L_k^-)^{k-1}| =|(L_k^-)^k| =|(L_k^+)^{k-1}|  =  |(L_k^+)^k| = 1-\frac k{2^k}\ge \frac 12.
\]
Hence $I$ intersects at most $9$ intervals from the definition of $w$, and we are in position to apply Claim A with $m=9$.

\vskip 1mm
{\bf Case 2}. $|I|> 4$. In this case, we may assume $|I|\in (2^{k_0}, 2^{k_0+1}]$ with some $k_0\ge 2$. We may further assume $k_0<N-10$ as otherwise
\[
\frac{w(I)}{|I|}\lesssim 2^{-N} w([0, 2^{N+1}])\simeq N \essinf\limits_{x\in I}w(x).
\]

\vskip 1mm
{\bf Case 2a}. $I\subset [0, 2^{k_0+10}]$. Then similarly to above,
\[
\frac{w(I)}{|I|}< 2^{-k_0} w([0, 2^{k_0+10}])\simeq k_0 \essinf\limits_{x\in I}w(x).
\]

\vskip 1mm
{\bf Case 2b}. $I\not\subset [0, 2^{k_0+10}]$. Then $I\subset [2^{k_0+9}, 2^{N+1}]$. Denote by $c_k$ the center of $L_k$.

\vskip 1mm
{\bf Case 2b-a}. $I$ contains some $c_k$ with $k\ge k_0+9$. Then the estimate is trivial since $I\subset (L_k^-)^1\cup (L_k^+)^1$ and
we apply Claim A with $m=2$.

\vskip 1mm
{\bf Case 2b-b}. $I$ does not contain any $c_k$. In this case we may assume $I\subset (c_{\ell}, c_{\ell+1})$ for some $k_0+8 \le \ell \le N$.

Suppose that $I=[a,b]$ and  $a\in (L_\ell^+)^{j}$ for some $j$. If $j \le  \ell- k_0-4$, then
\[
| (L_\ell^+)^{j+1}|= |L_\ell^+| 2^{-(j+1)}= \frac{2^\ell-\ell}{ 2^{j+2}}> 2^{k_0+1},
\]
so that $I$ will intersect at most $(L_\ell^+)^{j}$ and $(L_\ell^+)^{j+1}$ and we again apply Claim A with $m=2$.

If $j \ge \ell- k_0-3$, note that then
\[
I \subset \mathop{\cup}_{j=\ell-k_0-3}^\ell (L_\ell^+)^j \cup \mathop{\cup}_{i=\ell-k_0-2}^{\ell+1} (L_{\ell+1}^-)^i \cup I_{\ell+1}.
\]
Here $i\ge \ell-k_0-2$ since
\[
|(L_{\ell+1}^-)^{\ell-k_0-2} |= \frac{2^{\ell+1}-(\ell+1)}{2^{\ell-k_0-1 }}> \frac{2^{\ell}}{2^{\ell-k_0-1 }}=2^{k_0+1}\ge  \ell(I).
\]
Hence we have
\begin{align*}
\frac {w(I)}{|I| \essinf\limits_{x\in I} w(x)}&\le \frac{\sum\limits_{j=\ell-k_0-3}^\ell w((L_\ell^+)^j )+ \sum\limits_{i=\ell-k_0-2}^{\ell+1} w((L_{\ell+1}^-)^i )+ w(I_{\ell+1}) }{2^{k_0} 2^{\ell-k_0-3}}\\
&\lesssim  \frac{\sum\limits_{j=\ell-k_0-3}^\ell 2^j\cdot 2^{\ell-j}+  \sum\limits_{i=\ell-k_0-2}^{\ell+1} 2^i \cdot 2^{\ell+1-i} + (\ell+1)2^{\ell+2}}{2^\ell }
\lesssim \ell.
\end{align*}
It remains to consider the case $a\in I_{\ell+1}\cup L_{\ell+1}^-$. However, in this case we just need to discuss whether $b\in (L_{\ell+1}^-)^{j}$ with some  $j \le  \ell- k_0-3$ or not, which is completely similar. This completes the proof.
\end{proof}

\section*{Acknowledgement}
We would like to thank the anonymous referee for his/her careful reading that helped improving the presentation of the paper.


\begin{thebibliography}{99}
\bibitem{CRR} 
M. Caldarelli and I.P. Rivera-R\'ios, {\it A sparse approach to mixed weak type inequalities}, Math. Z. {\bf 296} (2020), no.1-2, 787--812.

\bibitem{CUIMPRR21}
D. Cruz-Uribe, J. Isralowitz, K. Moen, S. Pott and I.P. Rivera-R\'ios, {\it Weak endpoint bounds for matrix weights}, Rev. Mat. Iberoam. {\bf 37} (2021), no. 4, 1513--1538.

\bibitem{CUMP05}
D. Cruz-Uribe, J.M. Martell and C. P\'erez, {\it Weighted weak-type inequalities and a conjecture of Sawyer}, Int. Math. Res. Not. 2005, no. 30, 1849--1871.

\bibitem{DPTV}
K. Domelevo, S. Petermichl, S. Treil and A. Volberg, {\it The matrix $A_2$ conjecture fails, i.e. $3/2>1$}, preprint, available at arXiv:2402.06961.

\bibitem{IM19}
J. Isralowitz and K. Moen, {\it Matrix weighted Poincar\'e inequalities and applications to degenerate elliptic systems}, Indiana Univ. Math. J. {\bf 68} (2019), no. 5, 1327--1377.

\bibitem{IPRR}
J. Isralowitz, S. Pott and I.P. Rivera-R\'ios, {\it Sharp $A_1$ weighted estimates for vector valued operators}, J. Geom. Anal. {\bf 31} (2021), no. 3, 3085--3116.

\bibitem{H12}
T.P. Hyt\"onen, {\it The sharp weighted bound for general Calder\'on--Zygmund operators}, Ann. of Math. (2) {\bf 175} (2012), no. 3, 1473--1506.

\bibitem{HPV19}
T.P. Hyt\"onen, S. Petermichl and A. Volberg, {\it The sharp square function estimate with matrix weight},
Discrete Anal. 2019, Paper No. 2, 8 pp.

\bibitem{LOP19}
K. Li, S. Ombrosi and C. P\'erez, {\it Proof of an extension of E. Sawyer's conjecture about weighted mixed weak-type estimates}, Math. Ann. {\bf 374} (2019), no. 1-2, 907--929.

\bibitem{MW77}
B. Muckenhoupt and R.L. Wheeden, {\it Some weighted weak-type inequalities for the Hardy-Littlewood maximal function and the Hilbert transform}, Indiana Univ. Math. J. {\bf 26} (1977), no. 5, 801--816.

\bibitem{NPTV17}
F. Nazarov, S. Petermichl, S. Treil and A. Volberg, {\it Convex body domination and weighted estimates with matrix weights}, Adv. Math. {\bf 318} (2017), 279--306.

\bibitem{OPR}
S. Ombrosi, C. P\'erez and J. Recchi {\it Quantitative weighted mixed weak-type inequalities for classical operators}, Indiana Univ. Math. J. {\bf 65} (2016), no. 2, 615–640.

\bibitem{S85}
E. Sawyer, {\it A weighted weak type inequality for the maximal function}, Proc. Amer. Math. Soc. {\bf 93} (1985), no. 4, 610--614.

\end{thebibliography}
\end{document}